\documentclass{amsart}

\usepackage{amssymb}
\usepackage[all,cmtip]{xy}
\usepackage{graphicx}
\usepackage[usenames,dvipsnames]{color}
\usepackage{tikz}
\usepackage{hyperref}
\newtheorem{theorem}{Theorem}
\newtheorem{proposition}{Proposition}
\newtheorem{corollary}{Corollary}

\newtheorem{lemma}{Lemma}
\newtheorem*{problem}{Problem}

\newcommand\calT{{\mathcal T}}

\newcommand{\clos}{\operatorname{clos}}

\newcommand {\R}{\mathbb{R}}

\newcommand {\Z}{\mathbb{Z}}

\newcommand{\singsupp}{\operatorname{sing supp}}

\usepackage{url}
\usepackage{tikz}
\usetikzlibrary{calc}
\usetikzlibrary{decorations.markings}
\usetikzlibrary{through}
\definecolor{grey}{rgb}{0.8,0.8,0.8}
\definecolor{l.grey}{rgb}{0.9,0.9,0.9}
\definecolor{ll.grey}{rgb}{0.92,0.92,0.92,}
\definecolor{d.green}{rgb}{0,0.6,0.3}
\definecolor{l.green}{rgb}{0.8,0.99,0.69}
\definecolor{d.blue}{rgb}{0.2,0.2,0.7}
\definecolor{d.red}{rgb}{0.7,0,0}
\definecolor{myblue}{HTML}{92dcec}
\definecolor{l.blue}{rgb}{0.8,0.8,0.99}
\definecolor{ll.blue}{rgb}{0.4,0.6,0.6}
\definecolor{lll.blue}{rgb}{0.5,0.8,0.8}
\definecolor{d.orange}{rgb}{0.69,0.1,0}
\usepackage{xcolor}
\definecolor{light-gray}{gray}{0.8}

\begin{document}
\title{Hearing the shape of a triangle}
\author{Daniel Grieser}
\address{Institut f\"ur Mathematik, Carl von Ossietzky Universit\"at Oldenburg, 26111 Oldenburg, Germany}
\email{daniel.grieser@uni-oldenburg.de}

\author{Svenja Maronna}
\address{Institut f\"ur Mathematik, Carl von Ossietzky Universit\"at Oldenburg, 26111 Oldenburg, Germany}
\email{svenja.maronna@uni-oldenburg.de}
\maketitle

\begin{abstract}
In 1966 Mark Kac asked the famous question \lq Can one hear the shape of a drum?\rq. While this was later shown to be false in general, it was proved by C. Durso that one can hear the shape of a triangle. After an introduction to the general inverse spectral problem we will give a new proof of this fact. The central point of the argument is to show that area, perimeter and the sum of the reciprocals of the angles determine a triangle uniquely. This is proved using convexity arguments and the partial fraction expansion of $\sin^{-2}x$.
\end{abstract}

\section{Introduction}
The question \lq Can one hear the shape of a drum?\rq\ asked by Mark Kac in 1966  \cite{Kac:COHSD}  has attracted and inspired many mathematicians.
The methods used to understand this problem draw on diverse areas, for example partial differential equations, dynamical systems, group theory, number theory, and probability. In this article we will review some of the history and state of the art of the problem, and add a new twist to the story which leads to a curious elementary geometric problem about triangles, which we then solve.

Let us state the problem precisely. 
For a domain (bounded open set) $\Omega\subset\R^2$ consider the problem of finding a function $u$ on the closure of $\Omega$, vanishing at the boundary $\partial\Omega$, and a number $\lambda\in\R$ satisfying
\begin{align*}
 -\Delta u &= \lambda u
\end{align*}
in $\Omega$, where $\Delta  := \frac{\partial^2}{\partial x^2} + \frac{\partial^2}{\partial y^2}$ is the Laplace operator.
We call $\lambda$ a {\em Dirichlet eigenvalue} of $\Omega$ if there is a solution $u\not\equiv 0$. Multiplying the equation by $u$ and integrating by parts (i.e., using Green's identity) one sees that any eigenvalue  must be positive, and using basic techniques from PDEs and functional analysis one can show  (see \cite{Cha:ERG}) that the set of eigenvalues is an infinite discrete subset of $\R$ and that the eigenspace -- the set of solutions $u$ -- corresponding to each eigenvalue is finite dimensional. Hence one may write the eigenvalues as a sequence $0<\lambda_1\leq\lambda_2\leq \lambda_3\dots\to\infty$, where each eigenvalue is repeated according to the dimension of its eigenspace.

In this way a sequence of numbers $\lambda_1,\lambda_2,\dots$ is associated to each domain $\Omega$. This begs for mathematical investigation. Can we calculate the $\lambda_k$? No, except in very few cases, for example rectangles, the disk\footnote{Here 'calculate' is not to be taken literally: the eigenvalues are the squares of the zeroes of the Bessel functions}, certain triangles. Can we say anything interesting on how the eigenvalues depend on the shape of $\Omega$? Yes. This is the subject of the mathematical discipline called spectral geometry (see \cite{Ben:DE} for a short introduction and more references, and also \cite{Cha:ERG} and \cite{Lau:TFRSBEL}). We can also pose the {\em inverse problem}: Is the domain $\Omega$ determined uniquely by its eigenvalue sequence? Of course two congruent domains have the same eigenvalue sequence (we say they are {\em isospectral}), but do any two isospectral domains have to be congruent?
This is the question of Kac mentioned at the beginning, for the following reason:
Think of $\Omega$ as a drum, i.e. a membrane which is stretched over a wire frame in the shape of $\partial\Omega$. The membrane
can vibrate freely except that it is fixed at the boundary. When the drum vibrates you will hear a sound, which is composed of tones of various frequencies. These frequencies are the numbers $\gamma\sqrt{\lambda_k}$, where $\gamma$ is a constant depending on the material and tension of the drum.\footnote{This is an idealized physical model, for real drums the frequencies are slightly different due to non-linear effects and the influence of the resonance chamber.} So if you know $\gamma$ then in this sense you can 'hear' the eigenvalues $\lambda_k$. Without that knowledge you can still hear the quotients $\sqrt{\lambda_k/\lambda_1}$, which correspond to the musical intervals between the overtones and the fundamental tone of the drum's sound.

The problem may be easily generalized to higher dimensions and to compact Riemannian manifolds (with or without boundary). Already at Kac's time it was known that the answer is NO in the realm of Riemannian manifolds: Milnor had constructed two flat tori of dimension 16 which are isospectral but not isometric (the appropriate notion of congruence for Riemannian manifolds). So the question was whether there could also be a counterexample among domains in the plane.

It took 26 years to  reduce the dimension of counterexamples and make them fit into the plane. The first planar counterexamples were given  in 1992 by C. Gordon, D. Webb and S. Wolpert \cite{GorWebWol:OCHSD}.  Figure \ref{fig:counterexample} shows one of the first examples that were found. Since then many more examples of isospectral Riemannian manifolds, among them continuous families, have been found.
Recent surveys on these constructions are  \cite{GirTha:HSDMPAI} and \cite{GorPerSchue:IIMSTE}.

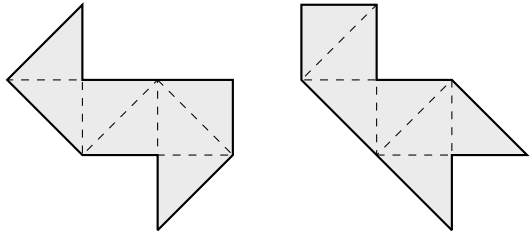
\begin{figure}
 
\begin{minipage}{0.3\textwidth}
\begin{tikzpicture}
 \draw[fill=ll.grey] (2,0) -- (3,1) -- (3,2) -- (1,2) -- (1,3) -- (0,2) -- (1,1) -- (2,1) -- (2,0);
 \draw[thick] (2,0) -- (3,1) -- (3,2) -- (1,2) -- (1,3) -- (0,2) -- (1,1) -- (2,1) -- (2,0);  
 \draw[dashed] (2,2) -- (3,1) -- (2,1) -- (2,2) -- (1,1) -- (1,2) -- (0,2);
\end{tikzpicture}
\end{minipage}
\begin{minipage}{0.3\textwidth}
\begin{tikzpicture}
 \draw[fill=ll.grey] (2,0) -- (2,1) -- (3,1) -- (2,2) -- (1,2) -- (1,3) -- (0,3) -- (0,2) -- (2,0);
 \draw[thick] (2,0) -- (2,1) -- (3,1) -- (2,2) -- (1,2) -- (1,3) -- (0,3) -- (0,2) -- (2,0);  
 \draw[dashed] (1,1) -- (2,1) -- (2,2) -- (1,1) -- (1,2) -- (0,2) -- (1,3);
\end{tikzpicture}
\end{minipage}
\caption{Two drums with the same overtones, see \cite{GorWebWol:OCHSD}. Isospectrality may be proved by transplantation, see \cite{BusConDoy:SPID}, \cite{Cha:DTSS} and \url{http://www.geom.uiuc.edu/docs/research/drums/planar/planar.html}, and \url{http://www.math.udel.edu/~driscoll/research/drums.html
} for pictures of eigenfunctions: For each triangle on the left one prescribes Euclidean motions to three triangles on the right. Then given any eigenfunction on the left drum, one transplants it to the right drum by moving the part of the eigenfunction on each left triangle to the right according to the given motions, and adding (inserting suitable $\pm$ signs) the functions obtained on each right triangle. The motions and signs can be chosen in such a way that the resulting function on the right drum is smooth across the dashed lines and hence an an eigenfunction with the same eigenvalue.}
\label{fig:counterexample}
\end{figure}
\section{What can you hear?}
Rather than focus on counterexamples to Kac's question let's be positive and ask which geometric properties of a domain or Riemannian manifold {\em can} be determined from its eigenvalue sequence. The indirect way in which the eigenvalues arise makes this seem a tough question to attack. However, there is a wonderful idea which helps us. It is the idea of transforms and traces. The two most important ones are the heat trace and the wave trace, corresponding to a sort of Laplace and Fourier transform of the eigenvalue sequence. More precisely, the heat trace is the function
\begin{equation}\label{heat kernel}
 h(t) = \sum_{k=1}^\infty e^{-\lambda_k t},\quad t>0 
\end{equation}
and the wave trace is 
\begin{equation}\label{wave kernel}
 w(t) = \sum_{k=1}^\infty \cos{\sqrt{\lambda_k} t},\quad t\in\R .
\end{equation}
The sum defining $h(t)$ converges for every $t>0$, and $h$ is a smooth function. The sum defining $w(t)$ never converges, but it can be made sense of in the sense of distributions, so $w$ is a distribution on $\R$.
For example if $\lambda_k=k^2$ and we sum over $k\in\Z$ then the Poisson summation formula gives
\begin{equation}
 \label{eqn:poisson}
 w(t)= \sum_{k\in\Z}\cos kt = 2\pi\sum_{l\in\Z} \delta_{2\pi l}(t)
\end{equation}
where $\delta_{2\pi l}$ is the delta distribution sitting at the point $2\pi l$
(to check this formally, simply calculate the Fourier series of the right hand side).
For this article, we will be sloppy about the distinction between functions and distributions.

So why are the functions $h$, $w$ useful for our problem? The reason is that there is a different way of understanding them, and this yields a relation to the geometry of $\Omega$. For $h$ this involves the heat equation
$$ (\partial_t - \Delta)v(t,x) = 0,\quad t>0, \ x\in\Omega$$
where $\partial_t:=\frac{\partial}{\partial t}$.
This equation has a unique solution for any initial data $v(0,x)=f(x)$, if we impose the boundary condition that $v(t,x)=0$ for all $t>0$ and $x\in\partial\Omega$. By separation of variables we obtain
$v(t,x) = \sum_{k=1}^\infty a_k e^{-\lambda_k t} u_k(x)$ where the $u_k$ form an orthonormal basis of $L^2(\Omega)$ of real valued eigenfunctions corresponding to the $\lambda_k$, and $a_k=\int_\Omega f(y)u_k(y)\,dy$.
In other words
$$v(t,x) = \int_\Omega H(t,x,y)f(y)\, dy$$ 
where $H(t,x,y) = \sum_{k=1}^\infty e^{-\lambda_k t} u_k(x) u_k(y)$. 
The function $H:(0,\infty)\times\Omega\times\Omega\to\R$ is called the heat kernel of $\Omega$, and since the $u_k$ are normalized in $L^2$ one sees that
\begin{equation}
 \label{eqn:heat trace}
 h(t) = \int_\Omega H(t,y,y)\, dy 
\end{equation}
This is the trace of the operator $e^{t\Delta}:f\mapsto v(t,\cdot)$, hence the name heat trace for $h$.
Now we observe that, for any fixed $y\in\Omega$, the function $(t,x)\mapsto H(t,x,y)$ is the solution of the heat equation with initial data $f(x)=\delta_y(x)$, that is, it describes the distribution of heat after time $t$, when initially there is a single hot spot at $y$. Although heat spreads at infinite velocity (that is, $H(t,x,y)>0$ for all $x$, no matter how small $t>0$), the value of $H(t,x,y)$ at $x=y$ for $t$ close to zero will be mostly influenced by the geometry of $\Omega$ near the point $y$. A precise analysis of the heat equation shows that for a Riemannian surface $\Omega$ without boundary
$$H(t,y,y)\sim t^{-1}\sum_{j=0}^\infty a_j(y) t^j\quad\text{ as }t\to 0 $$
where each $a_j(y)$ is a universal polynomial in derivatives of the Gauss curvature $K(y)$ of $\Omega$ at $y$. For example, $a_0(y)=\frac1{4\pi}$, $a_1(y) = \frac1{12\pi}K(y)$. 
If $\Omega$ has a boundary then its influence is felt only when the distance of $y$ to the boundary is of order at most $\sqrt t$, and in the integral \eqref{eqn:heat trace} this contributes extra terms involving the curvature of the boundary and terms involving the powers $t^{-1/2+j}$. In the case of planar domains with polygonal boundary there is no curvature, but the corners give a contribution, and this leads to the formula
$$ h(t) = a_0\, t^{-1} + a_{1/2}\, t^{-\frac12} + a_1 + O(e^{-\frac ct})
\quad\text{ as }t\to 0$$
for some constant $c>0$ where
$$ a_0 = \frac{A}{4\pi},\ a_{1/2} = - \frac P{8\sqrt\pi},\ a_1 = \frac1{24}\sum_{i}\left(\frac\pi{\alpha_i} - \frac{\alpha_i}\pi\right) $$
where $A$ is the area, $P$ the perimeter and the $\alpha_i$ are the interior angles of the polygon. This formula was first mentioned in \cite{McKSin:CEL}, the first published proof was given in \cite{BerSri:HERRWPB}.
In the case of the triangle we have $\sum_i \alpha_i=\pi$, so 
$a_1 = \frac\pi{24}\sum\limits_{i=1}^3 \frac1{\alpha_i}
 - \frac1{24}$. Therefore, if we know all the $\lambda_k$ then we know the function $h(t)$ and hence the coefficients $a_0,a_{1/2},a_1$, hence the area, the perimeter and the sum of the reciprocals of the angles of the triangle.
So we can hear these quantities. This motivates the following theorem:
\begin{theorem}\label{thm:triangle}
 A triangle is determined uniquely up to congruence by its area $A$, its perimeter $P$ and the sum $R$ of the reciprocals of its angles.
\end{theorem}
\begin{corollary}\label{cor}
 One can hear the shape of a triangle among all triangles.
\end{corollary}
That is, if we know that $\Omega$ is a triangle then the spectrum of $\Omega$ determines which triangle it is.

Before we embark on the proof of the theorem, let us digress and tell the remarkable story of the {\em wave kernel}, which is a much more powerful tool in spectral geometry than the heat kernel -- at the cost of harder technical issues in its analysis. The main idea, however, is easy to understand. The wave kernel
was used by  C. Durso in her first proof of  Corollary \ref{cor}, see \cite{Dur:SISPT}.

\subsection*{The wave kernel}
The wave trace $w(t)$ can be obtained in the same way as the heat trace, but starting with the wave equation
$$(\partial_t^2-\Delta)u(t,x)=0,\quad t\in\R,\ x\in\Omega $$
with initial data $u(0,x)=f(x)$, $(\partial_tu)(0,x)=0$ and boundary values $u(t,x)=0$ for all $t\in\R$, $x\in\partial\Omega$. This equation has a unique solution for each $f$, and it describes vibrations of $\Omega$, or propagation of waves on $\Omega$, with initial shape $f$. Again the solution can be written in the form 
$u(t,x) = \int_\Omega W(t,x,y)f(y)\, dy$ where now $W$ is a distribution, and
\begin{equation}
 \label{eqn:wave trace}
 w(t) = \int_\Omega W(t,y,y)\, dy, 
\end{equation}
the trace of the operator $\cos{t\sqrt{-\Delta}}:f\mapsto u(t,\cdot)$.

How can we learn anything about the function $W(t,x,y)$? It may help to think of $\Omega$ as a lake. At time $t=0$ we drop a stone into the lake at the place $y$ -- this corresponds to the initial condition $f(x)=\delta_y(x)$ -- and observe the resulting waves. In a linear water wave model, $x\mapsto W(t,x,y)$ is the lake's surface at time $t$.
Everyone knows what happens: A circular wave front centered at $y$ will form, its radius increasing linearly with $t$. When it reaches the boundary of the lake, it will be reflected. In our simple model there is no loss of energy and the wave will move on forever. The precise shape of the wave front can be described as follows: Starting at $y$ walk into any direction at speed 1. Always walk straight, except when you hit the boundary. In this case reflect off the boundary according to the law 'angle of incidence = angle of reflection'. The wave front at time $t$ is the set of points that you can reach in this way when walking for time $t$.

This helps us to understand the integrand $W(t,y,y)$ in \eqref{eqn:wave trace}: It will be large only for those times $t$ for which the wave front returns to $y$ after time $t$, i.e. for which there is a path\footnote{Here a 'path' is a succession of straight lines -- or geodesics -- obeying the law of reflection when hitting the boundary of $\Omega$. If $\Omega$ has corners, as for a triangle, then a path running into a corner can leave the corner in any direction.}
 from $y$ to $y$ of length $t$. A more careful analysis then shows that when integrated over $y$ many of these 'large' contributions cancel with neighboring paths due to oscillation. Only contributions from {\em closed paths}, that is those which return to $y$ in the same direction in which they started, are not cancelled in this way. To summarize, we arrive at the conclusion that $w(t)$ is large only for $|t|\in T$, where 
\begin{equation}
 \label{eqn:closed paths}
T = \{\text{lengths of closed paths in }\Omega\}\cup\{0\}.
\end{equation}
Here we also need to  count the 'instant' path of length zero.

The precise mathematical statement of this involves the notion of singular support of a distribution, i.e. the set where the distribution is not given by a smooth function. The wave front at time $t$ is precisely the singular support of the distribution $x\mapsto W(t,x,y)$. The result above translates into the statement 
$$ \singsupp w \subset \clos(T) $$
(see \cite{GuiMel:PSFMWB} for the rather technical proof, and \cite{Col:SLPGTYA} for a survey of the history of this theorem). If there is precisely one path for each $t\in T$ (up to reversal of direction) then the singular support is equal to $\clos(T)$, and it is conjectured that this equality is always true, but this is an open problem. In the example where $\Omega$ is a circle of length $2\pi$ (or alternatively, the interval $[0,2\pi]$ where we impose periodic boundary conditions), this can be seen explicitly: The eigenvalues are $k^2$, $k\in\Z$, and the wave kernel is \eqref{eqn:poisson} -- the numbers $2\pi l$ are precisely the lengths of closed paths in $\Omega$.

So we see that essentially, we can hear the set $T$ of lengths of closed paths on $\Omega$. This analysis can be refined substantially by analyzing the {\em kind} of singularities that the wave trace $w$ has at points of $T$. It turns out that the singularity at $t=0$ carries the same information as the full asymptotic expansion of the heat kernel at $t=0$. The other singularities yield additional information, und using this, one can prove that one can hear generic convex domains with analytic boundary and certain symmetries, see \cite{Zel:ISPADII}, \cite{HezZel:ISPAZSDRn}. As a final remark on this, we would like to mention the remarkable recent work \cite{ColGuiJer:SWTNCPLS} in which for the first time the behavior of $w$ at cluster points of $T$ was analyzed, in the special case of a disk. A recent survey on inverse spectral results obtained using trace formulae and related methods is \cite{DatHez:IPSG}.

To end this section, let us explain Durso's proof that one can hear the shape of a triangle. 

It is classical that in an acute triangle $\Omega$, there is a unique shortest closed path, and it is given by the triangle formed by the base points of the three altitudes of $\Omega$, see Figure \ref{fig:fagnano}.
Therefore, one can hear the length of this path. Durso shows that, in the case of an obtuse or right-angled triangle, the shortest closed path is the shortest altitude, traversed up and down, and that the wave trace $w$ is singular at $l_0$, the length of this path (this is the hard analytical part of the proof). So one can hear $l_0$. Then she shows by an elementary geometric argument that any triangle is determined uniquely by area, perimeter and the length of its shortest closed path.

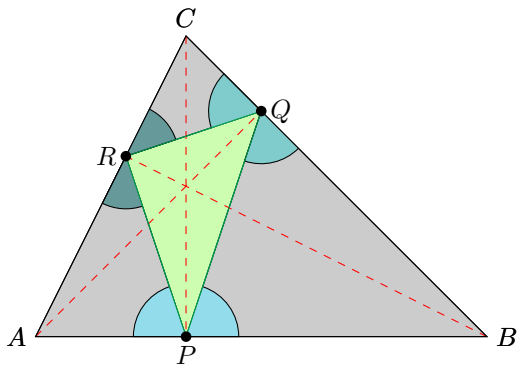
\begin{figure}
 \begin{tikzpicture} 
   \coordinate (a) at (0,0) ;
   \coordinate (b) at (6,0) ;
   \coordinate (c) at (2,4) ; 
   \coordinate[label=below:$P$] (p) at (intersection cs: first line={(a)--(b)},
		second line={(c)--($(a)!(c)!(b)$)}) circle (2pt);
	 \coordinate[label=right:$Q$] (q) at (intersection cs: first line={(b)--(c)},
		second line={(a) -- ($(b)!(a)!(c)$)}) circle (2pt);		 
	 \coordinate [label=left:$R$] (r) at (intersection cs: first line={(c)--(a)},
		second line={(b) -- ($(a)!(b)!(c)$)}) circle (2pt);
		
	 \draw [fill=grey](a)node[left]{$A$} -- (b) node[right]{$B$}-- (c) node[above]{$C$}-- cycle;
	 \draw[fill=l.green] (p) -- (q) -- (r) -- (p) ;
	 \draw[d.green] (p)--(q)--(r) -- (p); 
   \draw[dashed][red] (a) -- ($(b)!(a)!(c)$); 
   \draw[dashed][red] (b) -- ($(a)!(b)!(c)$); 
   \draw[dashed][red] (c) -- ($(a)!(c)!(b)$);
   \draw[fill=myblue] (p) -- +(0:0.7cm) arc (0:71:0.7cm);
   \draw[fill=myblue] (p) -- +(180:0.7cm) arc (180:109:0.7cm);
   \draw[fill=ll.blue] (r) -- +(18.5:0.7cm) arc (18.5:63.5:0.7cm); 
   \draw[fill=ll.blue] (r) -- +(243.5:0.7cm)arc (243.5:288:0.7cm);
   \draw[fill=lll.blue] (q) -- +(135:0.7cm)arc (135:199:0.7cm);
   \draw[fill=lll.blue] (q) -- +(251.5:0.7cm)arc (251.5:315:0.7cm);
   \draw[d.green] (p)--(q)--(r) -- (p);  
   \draw(a)node[left]{$A$} -- (b) node[right]{$B$}-- (c) node[above]{$C$}-- cycle;
      
      \fill (intersection cs:
		first line={(c)--(a)},
		second line={(b) -- ($(a)!(b)!(c)$)}) circle (2pt); 
		\fill (intersection cs:
		first line={(b)--(c)},
		second line={(a) -- ($(b)!(a)!(c)$)}) circle (2pt);
		\fill (intersection cs:
		first line={(a)--(b)},
		second line={(c)--($(a)!(c)!(b)$)}) circle (2pt);
    
\end{tikzpicture}
 \caption{In an acute triangle $\triangle ABC$, how should one choose points $P,Q,R$ on each side so that the triangle $\triangle PQR$ has minimal perimeter? The answer: Choose the base points of the altitudes of $\triangle ABC$. The resulting triangle is called the {\em Fagnano triangle}. There is a clever proof of this fact using reflections of the side $AB$ across the sides $AC$ and $BC$. Also, by a standard variational argument it follows that the circumference of the Fagnano triangle obeys the law of reflection, so an ideal billiard ball on the billiard table $\triangle ABC$ will run forever along this line. 
 }
 \label{fig:fagnano}
\end{figure}
 
\section{A theorem about triangles} \label{proof}
We now return to the proof of Theorem \ref{thm:triangle}. This is a rather peculiar statement: Have you ever heard of {\em reciprocals of angles}? There does not seem to be any geometric meaning to this, and our proof draws on classical analysis rather than geometry. 
Note that in contrast to Durso's proof, our proof only uses tools known in the 1960s.

First, let us remark that it is quite clear that the three quantities $A,P,R$ determine a triangle {\em up to finitely many choices}. This follows easily from the fact that the space of triangles $\calT$ is three dimensional (for example, it may be parametrized by the side lengths), that the functions $A,P,R$ on $\calT$ are analytic and independent (in the sense that none of them can be expressed as a function of the other two; independence in a stronger sense will be proved in Lemma \ref{lem1} below), and that $R$ is a proper function on $\calT/\R_{>0}$, the quotient of $\calT$ by scalings: $R$ tends to infinity when one of the angles tends to zero, which is the only way to leave all compact subsets of $\calT/\R_{>0}$.
However, just as prescribing the lengths of two sides and an angle not enclosed by them determines a triangle only up to two choices, it is not obvious why 
there should be only one triangle with any given $A,P,R$. Of course it is not hard to check this numerically, but it is far from obvious how to prove it analytically. 
This is what we shall do.

\begin{figure}
 \begin{tikzpicture}
		\coordinate (a) at (0,0);
   	\coordinate (b) at (0:5);
   	\coordinate (c) at (intersection cs:first line={(5,0)-- +(145:20)}, second line={(0,0)-- +(75:10)});
		\coordinate (h1) at (intersection cs: first line={ (b)--(c)}, second line={(a)--(37.5:20)});
		\coordinate (h2) at (intersection cs: first line={ (a)--(c)}, second line={(5,0)-- +(162.5:20)});
		\coordinate [label=left: $M$] (m) at (intersection cs: first line={(a)--(h1)}, second line={(b)--(h2)}) circle (2pt);
		\draw [fill=ll.grey](a)node[left]{$A$}--(b)node[right]{$B$}--(c)node[above]{$C$}--(a);
		\draw [dashed][d.blue](m)--(a);
		\draw [dashed][d.blue](m)--(b);
		\draw [dashed][d.blue](m)--(c);
		\draw [dashed][red](m)--($(b)!(m)!(c)$);
		\draw [dashed][red](m)--($(a)!(m)!(b)$);
		\draw [dashed][red](m)--($(c)!(m)!(a)$);
		
		\filldraw[fill=green!20!white, draw=green!50!black]
		(0,0) -- (5mm,0mm) arc (0:37.5:5mm)node[right]{\textcolor{green!50!black}{$ ~\frac{\alpha}{2}$}} -- cycle;
		\coordinate (s2) at (intersection cs: first line={ (a)--(b)}, second line={(m)--($(a)!(m)!(b)$)});
		\node [d.orange][draw,circle through=(s2)] at (m) {};
		\node[red] at (19:1.5){$r$} {};
		\coordinate (s1) at (intersection cs: first line={ (a)--(c)}, second line={(m)--($(a)!(m)!(c)$)});
		\draw (s2) -- +(180:0.2cm) arc (180:90:0.2cm);
		\fill (4:1.15) circle (0.7pt);
\end{tikzpicture}\\

 \caption{Proof of \eqref{eqn:AP}: $A=r\frac P2$ and $\cot\frac\alpha2 + \cot\frac\beta2 + \cot\frac\gamma2=\frac1r\frac P2$}
 \label{fig:triangle-formula}
\end{figure}
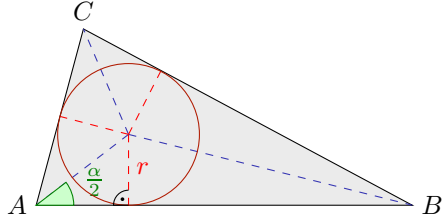
We denote the angles of the triangle by $\alpha,\beta,\gamma$.
We use the following formula from triangle geometry\footnote{We are grateful to Richard Laugesen for pointing out this identity. Amazingly, both sides are also equal to the product $\cot\frac\alpha2  \cot\frac\beta2  \cot\frac\gamma2$; it's a nice little exercise in addition theorems to prove this.}, see Figure \ref{fig:triangle-formula}:
\begin{equation}
 \label{eqn:AP}
 \frac{P^2}{4A} = \cot\frac\alpha2 + \cot\frac\beta2 + \cot\frac\gamma2
\end{equation}
This allows us to work exclusively with angles. We will prove:\begin{proposition}\label{prop}
 A triple $(\alpha,\beta,\gamma)$ of positive real numbers satisfying $\alpha+\beta+\gamma=\pi$ is uniquely determined, up to ordering, by the values of
\begin{align}
 f(\alpha,\beta,\gamma) &= \cot\frac\alpha2 + \cot\frac\beta2 + \cot\frac\gamma2 \\
 g(\alpha,\beta,\gamma) &= \frac1\alpha + \frac1\beta + \frac1\gamma \,. 
\end{align}
\end{proposition}
Theorem  \eqref{thm:triangle} follows directly from this: if the area $A$ and perimeter $P$ are given then the angles are determined by equation \eqref{eqn:AP} and 
the Proposition, so the triangle is determined up to dilation. Then the given area fixes the dilation factor.

So it remains to prove the Proposition. One way to proceed would be to eliminate one of the variables, say $\alpha$, using the relation $\alpha=\pi-\beta-\gamma$, then eliminate another variable (say $\beta$) from the given value of $g$ by solving a quadratic equation, then plug the expressions for $\alpha$ and $\beta$ into $f$ and investigate the resulting equation for $\gamma$. But this is horrible! Even if it works, it is ugly mathematics. If nothing else, the beautiful symmetry present in the statement of the Proposition is lost. 

Symmetry is a treasure. One should keep it and use it as long as possible. This is what we shall do.

\begin{proof}[Proof of the Proposition]
 Let $D=\{(\alpha,\beta,\gamma): \alpha,\beta,\gamma>0,\ \alpha+\beta+\gamma=\pi\} \subset \R_{>0}^3$ where $\R_{>0}=(0,\infty)$. We think of points of $D$ as 'marked triangles up to dilation', where 'marked' means that we have named the angles in a certain order. The set $D$ is (the interior of) a triangle itself -- the triangle cut out of the plane $\alpha+\beta+\gamma=\pi$ by the positive octant, see Figure \ref{fig:angles}. Points on the dashed lines correspond to isosceles triangles, the center $e$ corresponds to the equilateral triangle. 
 Let us call a point which does not lie on a dashed line a {\em non-isosceles point}. The non-isosceles points form six connected subsets, which we call {\em chambers}.
  The dashed lines are also lines of symmetry:  If we pick a non-isosceles  point  and reflect it step by step across all dashed lines, we obtain six points, one in each chamber; these six points correspond to the same triangle, with angles named in different orders. Each chamber corresponds to one ordering of the angles, for example the lower left chamber to the ordering $\alpha>\beta>\gamma$, or $\alpha\geq\beta\geq\gamma$ when we include its dashed boundary parts.
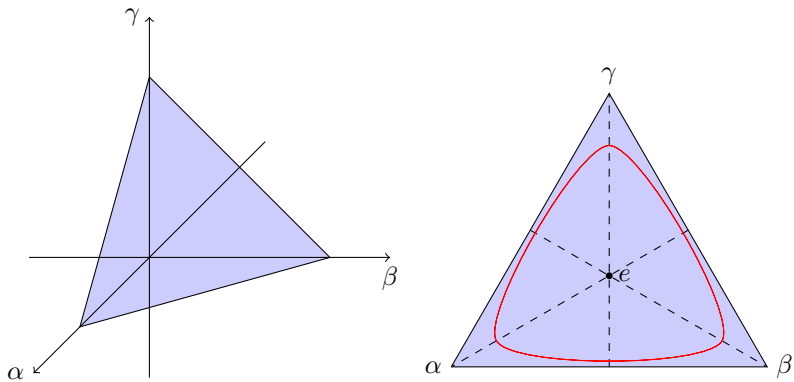
\begin{figure}
\begin{tikzpicture}[scale=0.8]
   
	\draw [fill= l.blue](3,0,0) -- (0,3,0) -- (0,0,3) -- (3,0,0);

\draw[->] (0,-2,0) -- (0,4,0) node[left]{$\gamma$};
   \draw[->] (-2,0,0) -- (4,0,0) node[below]{$\beta$};
   \draw[->] (0,0,-5) -- (0,0,5) node[left]{$\alpha$};
\end{tikzpicture}
 \begin{tikzpicture}[scale=1.0]
 \begin{scope}[thin,black]

	\coordinate (P0) at (0,0);
	\coordinate (P1) at (60:4.2);
	\coordinate (P2) at (0:4.2);
	
	\draw[fill=l.blue] (0,0) node[left]{$\alpha$} --+ (60:4.2) node[above]{$\gamma$} --+ (0:4.2) node[right]{$\beta$} --(0,0);

	\draw[dashed](2.1,0) -- (P1);
	\draw[dashed](60:2.1) -- (P2);
	\draw[dashed](30:3.6373) -- (P0);
	\draw (2.1,1.2124) node[right]{$e$};
	\path[fill] (2.1,1.2124) circle (0.3ex);

\draw[red,smooth cycle] plot coordinates{(0.6,0.34641) (3.6,0.34641) (2.1, 2.944486) (0.6,0.34641) (3.6,0.34641) (2.1, 2.944486)};

 \end{scope}
\end{tikzpicture}
\caption{The space of angles of a triangle, and a level line of $g$}
\label{fig:angles}
\end{figure}

The idea of the proof is to show that the level sets of the function $g$ are convex curves, see Figure \ref{fig:angles}, and that $f$ is strictly monotone along the part of any one of these curves lying in one chamber.

\begin{lemma}\label{lem1} \mbox{}
Consider the functions $f,g$ and $h(\alpha,\beta,\gamma)=\alpha+\beta+\gamma$ on the positive octant $\R_{>0}^3$. 
\begin{enumerate}
 \item[a)] The function $g$ is strictly convex on $\R_{>0}^3$.
 \item[b)]  The gradients $\nabla f, \nabla g, \nabla h$ are linearly independent at all non-isosceles points of $D$. 
\end{enumerate}
\end{lemma}
Let us finish the proof of the Proposition and then return to prove Lemma \ref{lem1}. 
The strict convexity of $g$ implies that the sublevel set $G_{\leq s}=\{p\in \R_{>0}^3:g(p)\leq s\}$ is strictly convex for any $s>0$, with boundary the level surface $G_s=\{p\in \R_{>0}^3:g(p) = s\}$. Furthermore, these sets are symmetric under all permutations of the coordinates. These properties then also hold for the intersections of the sublevel and level sets with the plane $\alpha+\beta+\gamma=\pi$. Since $g(p)\to\infty$ when $p$ approaches the boundary of $D$ (i.e. when at least one of the angles tends to zero), it follows that the sets $G_s\cap D$ are either closed curves in the interior of $D$ which encircle the point $e$, or the point $e$, or empty. Since the equilateral triangle has $g(e)=\frac9\pi$, the first case corresponds to $s>\frac9\pi$.

In particular, we see that the point $e$ is already determined by the value of $g$ alone\footnote{This can also be seen from the arithmetic-harmonic mean inequality: $3\left(\frac1\alpha+\frac1\beta+\frac1\gamma\right)^{-1} \leq \frac{\alpha+\beta+\gamma}3$ with equality iff $\alpha=\beta=\gamma$. 
}.

Now consider any level curve $G_s\cap D$ with $s>\frac9\pi$. Consider the arc of the curve running inside one chamber, with endpoints $p$, $q$ corresponding to isosceles triangles. Our proof will be complete if we can show that $f$ is strictly monotone along this part of the curve. 

Suppose $f$ was not strictly monotone. Then there would be a point $r$ on this arc, different from $p$ and $q$, where $f$ is stationary, that is, the derivative of $f$ along the arc vanishes at $r$. By the Lagrange multiplier theorem this would mean that $\nabla f (r)$ is a linear combination of $\nabla g(r)$ and $\nabla h(r)$. But this would be  a contradiction to part b) of the Lemma. This completes the proof of the Proposition.
\end{proof}
\begin{proof}[Proof of Lemma \ref{lem1}]
a)
 The Hessian (matrix of second derivatives) of $g$ is the diagonal matrix with entries 
 $ \frac2{\alpha^3}$, $\frac2{\beta^3}$, $\frac2{\gamma^3}$ on the diagonal. 
This is clearly positive definite for all $(\alpha,\beta,\gamma)\in\R_{>0}^3$, and this implies that $g$ is strictly convex.

b)
We have 
\[
\renewcommand\arraystretch{1.2}
\nabla f = -\frac12 
\begin{pmatrix}
 \frac1{\sin^2\frac\alpha2} \\
  \frac1{\sin^2\frac\beta2} \\
   \frac1{\sin^2\frac\gamma2}
\end{pmatrix}
,\ \nabla g = -
\begin{pmatrix}
 \frac1{\alpha^2}  \\
 \frac1{\beta^2} \\
 \frac1{\gamma^2}
\end{pmatrix}
,\ \nabla h = 
\begin{pmatrix}
 1 \\ 1 \\ 1
\end{pmatrix}
\]
Suppose there is a non-isosceles point $(\alpha,\beta,\gamma)$ (i.e. these numbers are pairwise different) and numbers $R,S,T$, not all zero, with $R\nabla f + S\nabla g + T\nabla h = 0$. This would mean that the function $$F(y)= -\frac R2 \frac1{\sin^2\frac y2} - S\frac1{y^2} + T$$ had three different zeroes in the interval $(0,\pi)$, namely $y=\alpha$, $y=\beta$ and $y=\gamma$. In order to show that this cannot happen we prove that the function $F$ is a non-zero constant, or strictly monotone, or strictly concave or convex on this interval, depending on the values $R,S,T$. We use the following fact, proved below:
\begin{lemma}\label{lem2}
 The function $\displaystyle G(x) = \frac1{\sin^2x}-\frac1{x^2}$ is strictly increasing and strictly convex on the interval $(0,\pi)$.
\end{lemma}
This lemma implies that the function $G_C(x)=\frac 1{\sin^2 x} - \frac C{x^2}$ is, on the interval $(0,\pi)$, strictly increasing for $C\geq 1$ and strictly convex for $C\leq 1$, since $G_C(x) = G(x) + \frac{1-C}{x^2}$ and the function $\frac{1-C}{x^2}$ is increasing for $C>1$ and convex for $C<1$.
Now clearly for any values of $R,S,T$ we can write $F(y)$ as a constant multiple of $G_C(\frac y2)$, for some $C$, plus a constant, and the claim follows.
\end{proof}
\begin{proof}[Proof of Lemma \ref{lem2}]
First note that this is non-trivial: It is easy to check that both $\frac1{\sin^2x}$ and $\frac1{x^2}$ have positive second derivative whenever they are defined, hence are convex, but it is not clear why their difference should be convex. However, things become very transparent when we use the series representation (partial fraction expansion)
$$ \frac1{\sin^2x} = \sum_{k=-\infty}^\infty \frac1{(x-k\pi)^2} $$
which follows from the well-known partial fraction expansion of the cotangent by differentiation.
This yields $G(x) =  \sum\limits_{k\neq 0} \frac1{(x-k\pi)^2}$. Now every summand $\frac1{(x-k\pi)^2}$ is strictly convex on $(0,\pi)$ since the function $\frac1{x^2}$ is strictly convex on both half lines $x<0$ and $x>0$, so $G$ is strictly convex. Furthermore, the series shows that $G$ is  regular at $x=0$, and it is also even, so $G'(0)=0$. Combined with strict convexity this implies that $G$ is strictly increasing on the interval $(0,\pi)$, which was to be shown.
%
%
\end{proof}

\section{Further remarks}\label{sec3}
Let us take another look at the proof of Theorem \ref{thm:triangle}, from a slightly different perspective.
Proposition \ref{prop}, which implies Theorem \ref{thm:triangle} by elementary triangle formulas, may be restated as saying that the map $\Phi=(f,g): D\to\R^2$ is injective on the closure in $D$ of each chamber. The proof of injectivity has two ingredients: First, Lemma \ref{lem1} b), which may be restated as saying that the differential of the map $\Phi$ is invertible in the chamber and hence, by the inverse mapping theorem, that $\Phi$ is locally injective everywhere, that is, every point of the chamber has a neighborhood on which $\Phi$ is injective. Second, the convexity of Lemma \ref{lem1} a) allows to infer global injectivity from this local statement. Finally, the analytic core of the whole argument is Lemma \ref{lem2} which is used in the proof of Lemma \ref{lem1} b). We now take another look at this.

\subsection*{A different proof of Lemma \ref{lem2}}  While the proof using the partial fraction representation is very elegant, you might wonder if there  is a more pedestrian way to prove convexity of $G$. Indeed there is. Here is a sketch. It was our first proof of this result, and it is the result of the Bachelor's thesis of the second author.
A short calculation gives $\frac12 G''(x) = \frac3{\sin^4 x} - \frac2{\sin^2x} - \frac3{x^4}$. We need to show that this is positive (here and in the sequel we always assume $x>0$). This is equivalent to the inequality
\begin{equation}
\label{eqn convex} 
 3\sin^4x + 2x^4\sin^2x \overset{!}{<} 3x^4
\end{equation}
How can one prove an inequality involving trigonometric functions and polynomials? Maybe your first idea is to use the well-known inequality $\sin x< x$ to get rid of the sines. But clearly this does not help since $3x^4 + 2x^4\cdot x^2 > 3x^4$. How can we do better? 

Recall where the inequality $\sin x<x$ comes from: $x$ is the first term in the Taylor series of $\sin x$, the next term is negative. Of course this is not a proof, but it's the core idea, which can be turned into a proof as follows: 
The function $f(x)=x-\sin x$ vanishes at $x=0$ and has derivative $f'(x)=1-\cos x$, which is always non-negative, and is positive for small positive $x$. Thus, $x-\sin x>0$ for all positive $x$ follows by integration: $f(x)=\int_0^x f'(t)\,dt > 0$.

So in order to prove \eqref{eqn convex} we can try to use a better estimate for $\sin x$ by using more terms from its Taylor series. We have the estimate
\begin{equation}
\label{eqn sin} 
 \sin x < x - \frac{x^3}6 + \frac{x^5}{120}
\end{equation}
This can proved in the same way as $\sin x<x$: The function $f(x)=x - \frac{x^3}6 + \frac{x^5}{120} -\sin x$ satisfies $f(0)=f'(0)=f''(0)=f'''(0)=f^{(4)}(0)=0$ and $f^{(5)}(x) = 1-\cos x\geq 0$, and $>0$ for small positive $x$. Integrating we obtain $f^{(4)}(x)=\int_0^x f^{(5)}(t)\,dt > 0$, then integrating again we get $f'''(x)>0$ and so forth, until we obtain $f(x)>0$ for all $x>0$.\footnote{Instead we could have used Taylor's formula with remainder for the function $g(x)=\sin x$:
$$ g(x)= \sum_{k=0}^4 \frac{x^k}{k!} + \frac1{4!}\int_0^t (x-t)^4 g^{(5)}(t)\,dt$$
which using $g^{(5)}(t)=\cos t \leq 1$ (and $<1$ for small positive $t$) and
$\int_0^t (x-t)^4\,dt = \frac15 x^5$ yields the same result. Yet another proof uses Leibniz' criterion for the Taylor series $x-\frac{x^3}{3!}+\frac{x^5}{5!} +\dots$ of $\sin x$, which is alternating. The terms after the fifth power are monotonically decreasing in absolute value if $\frac{x^{2n+1}}{(2n+1)!}<\frac{x^{2n-1}}{(2n-1!}$ for $n\geq4$, which is equivalent to $x^2<2n(2n+1)$, hence true for $x<\sqrt{72}$. Since the first omitted term after $\frac{x^5}{5!}$  is negative, we get that the sum of the series, which is $\sin x$, is less than  $x-\frac{x^3}{3!}+\frac{x^5}{5!}$, at least for $x<\sqrt{72}$. Since $\sqrt{72}>\pi$, this is enough for our purpose.
}

We now plug \eqref{eqn sin} into the left hand side of \eqref{eqn convex}. A rather tedious calculation shows that the result, which starts as $3x^4 - \frac1{15}x^8+\dots$, is less than $3x^4$ for $x<4$. 
The main point is that the second term is negative. Among the higher terms some are positive, but they can easily be estimated against the negative ones.

\subsection*{A few open problems}
The way in which the Dirichlet eigenvalues determine the triangle is somewhat indirect: First one constructs the heat kernel $h$, see \eqref{heat kernel}, and then one considers the coefficients in its asymptotic expansion to prove the result. In particular, one needs to know (asymptotic information on) {\em all}  the eigenvalues for this. It is natural to ask whether already a finite number of eigenvalues, ideally only three, suffice to determine the triangle. 
\begin{problem}
 Do the first three Dirichlet eigenvalues $\lambda_1,\lambda_2,\lambda_3$ determine a triangle?
\end{problem}
Numerical evidence was provided in \cite{AntFre:ISPET} that this is true -- but that the corresponding statement for $\lambda_1,\lambda_2,\lambda_4$ is false. However, no proof of this is known. As a partial result in this direction it is proved in \cite{ChaDeT:HST} that for each $\varepsilon>0$ there is a number $N$ so that $\lambda_1,\dots,\lambda_N$ determine a triangle uniquely among all triangles whose angles are all greater than or equal to $\varepsilon$.

\begin{problem}
 Is there a closed path (not hitting a corner) on every triangle?
\end{problem}
For acute triangles the answer is yes, see Figure \ref{fig:fagnano}. The problem is open for general obtuse triangles.

\begin{problem}
 Does the second Neumann eigenfunction on an obtuse triangle have its extrema on the boundary?
\end{problem}
This is conjectured to be true, and is a special case of the {\em hot spots conjecture}. See the recent discussion on Polymath, \cite{Polymath:HSC}. 

Let us mention two other open questions on the inverse spectral problem.
\begin{problem}
Can one hear the shape of a convex polygon? Can one hear the shape of a domain $\Omega\subset\R^2$ with smooth boundary?
\end{problem}
\noindent We emphasize that the answer is no when convexity or smoothness is not required: All known counterexamples to 'Can one hear the shape of drum?' are non-convex polygons, cf. Figure \ref{fig:counterexample}.

\bibliographystyle{plain}
\bibliography{dglib}
\end{document}